\documentclass[a4paper,12pt, times]{article}
\usepackage[utf8]{inputenc}

\usepackage{mathrsfs}
\usepackage{graphicx}
\usepackage{enumerate}
\usepackage{multicol}
\usepackage{color}
\usepackage{amsmath,amssymb,amscd}
\usepackage{amsthm}

\usepackage{times}

\theoremstyle{plain}
\newtheorem{thm}{Theorem}
\newtheorem{lem}{Lemma}
\newtheorem{cor}{Corollary}

\theoremstyle{definition}
\newtheorem{dfn}{Definition}
\newtheorem{ex}{Example}

\theoremstyle{remark}
\newtheorem{rem}{Remark}

\title{On functional equations characterizing derivations: methods and examples}
\author{Eszter Gselmann, Gergely Kiss and Csaba Vincze}

\begin{document}

\maketitle

\begin{abstract}
Functional equations satisfied by additive functions have a special interest not only in the theory 
of functional equations, but also in the theory of (commutative) algebra because the fundamental notions such as derivations and automorphisms are additive functions satisfying some further functional equations as well. It is an important question that how these morphisms can be characterized among additive mappings in general. 

The paper contains some multivariate characterizations of higher order derivations. 
The univariate characterizations are given as consequences by the diagonalization of the multivariate formulas. 
 This method allows us to refine the process of computing the solutions of univariate functional equations of the form 
\[
 \sum_{k=1}^{n}x^{p_{k}}f_{k}(x^{q_{k}})=0, 
\]
where $p_k$ and $q_k$ ($k=1, \ldots, n$) are given nonnegative integers and the unknown functions $f_{1}, \ldots, f_{n}\colon R\to R$ are supposed to be additive on the ring $R$. It is illustrated by some explicit examples too. 

As another application of the multivariate setting we use spectral analysis and spectral synthesis in the space of the additive solutions 
to prove that it is spanned by differential operators. 
The results are uniformly based on the investigation of the  multivariate version of the functional equations.
\end{abstract}

\section{Introduction}

Functional equations satisfied by additive functions have a rather extensive literature
\cite{KisLac25}, \cite{KisVarVin15}, \cite{KisVin17}, \cite{KisVin17a}, \cite{Var10}, \cite{VarVin09}. They appear not only in the theory 
of functional equations, but also in the theory of (commutative) algebra \cite{Kha91}, \cite{Kuc09}, \cite{Rei98}, \cite{UngRei98}, \cite{ZarSam75}. 
It is an important question that how special morphisms can be characterized among additive mappings in general. 
This paper is devoted to the case of functional equations characterizing derivations. It is motivated by some recent results \cite{Eba15}, \cite{Eba17}, \cite{Eba18}
due to B. Ebanks. We are looking for solutions of functional equations of the form 
\begin{equation}
\label{Ebanks}
 \sum_{k=1}^{n}x^{p_{k}}f_{k}(x^{q_{k}})=0, 
\end{equation}
where $p_k$ and $q_k$ ($k=1, \ldots, n$) are given nonnegative integers and the unknown functions $f_{1}$, $\ldots$, $f_{n}\colon R\to R$ are supposed to be additive on the ring $R$. According to the homogeneity of additive functions \cite[Lemma 2.2]{Eba15} it is enough to investigate equations with constant pairwise sum of the powers, i.e. 
$$p_1+q_1=\ldots=p_n+q_n.$$
The general form of the solutions is formulated as a conjecture in \cite[Conjecture 4.15]{Eba15}. The proof can be found in \cite{Eba17} 
by using special substitutions of the variable and an inductive argument. We have a  significantly different starting point by following the method of free variables. The paper contains some multivariate characterizations of higher order derivations in Section 2. 
The basic results of \cite{Eba15}, \cite{Eba17} (univariate characterizations of higher order derivations) are given as consequences by the diagonalization of the multivariate formulas, see Section 3. 
The method allows us to refine the process of computing the solutions of functional equations of the form (\ref{Ebanks}). The examples (Examples I, Examples II) illustrate that the multivariate version of the functional equations provides a more effective and subtle way to determine the structure of the unknown functions. Especially, functional equations with missing powers (Examples II) can be investigated in this way to avoid 
formal (identically zero) terms in the solution. As a refinement of Ebanks' method we follow the main steps such as  
\begin{enumerate}[(i)]
\item The formulation of the multivariate version of the functional equation. 
\item The substitution of value $1$ as many times as the number of the missing powers (due to the symmetry of the variables there is no need to specify the positions for the substitution).
\item The application of Theorem \ref{Thm_mainmultivariate}/Corollary \ref{Thm_main}.
 \end{enumerate}
The second step decreases the homogeneity degree of the functional equation to keep only non-identically zero terms in the solutions. Therefore it can be easily seen how the number of the nonzero coefficients is related to the maximal order of the solution \cite{EbaRieSah}. 

In Section 4 we present another approach to the problem. The application of the spectral ana\-ly\-sis and the spectral synthesis is a 
general method to find the additive solutions of a functional equation, see \cite{KisLac25}. 
It is a new and important trend in the theory of functional equations; see e.g. \cite{KisVin17} and \cite{KisVin17a}. 
Although the domain should be specified as a finitely generated subfield over the rationals, 
the multivariate version of the functional equation generates a system of functional equations concerning the translates of the original solutions 
in the space of complex valued additive functions restricted to the multiplicative subgroup of the given subfield. 
Taking into account the fundamental result \cite[Theorem 4.3]{KisLac25} such a closed translation invariant linear 
subspace of additive functions contains automorphism solutions (spectral analysis) and the space of the solutions is 
spanned by the compositions of automorphisms and differential operators (spectral synthesis). 
All functional equations in the paper are also discussed by the help of the spectral analysis and the spectral synthesis. 
We prove that the automorphism solutions must be trivial (identity) and, consequently, 
the space of the solutions is spanned by differential operators. According to the results presented in Section 2 and Section 3, 
the investigation of the detailed form of the differential operator solutions is omitted. 
However, Subsection 4.4 contains an alternative way to prove Theorem \ref{Thm_mainmultivariate}/Corollary \ref{Thm_main} 
in the special case $n=3$. 
The proof uses a descending process instead of the inductive argument.  

In what follows we summarize some basic theoretical facts, terminology and notations.

\subsection*{Derivations}

For the general theory of derivations we can refer to Kuczma \cite{Kuc09}, see also Zariski--Samuel \cite{ZarSam75} and Kharchenko \cite{Kha91}.

\begin{dfn}\label{D2.1.1}
Let $Q$ be a ring and consider a subring $P\subset Q$.
A function $f:P\rightarrow Q$ is called a \emph{derivation}\index{derivation} if it is additive,
i.e.
\[
f(x+y)=f(x)+f(y)
\quad
\left(x, y\in P\right)
\]
and also satisfies the so-called \emph{Leibniz rule}\index{Leibniz rule}
\[
f(xy)=f(x)y+xf(y)
\quad
\left(x, y\in P\right). 
\]
\end{dfn}

\begin{ex}
Let $\mathbb{F}$ be a field, and $\mathbb{F}[x]$
be the ring of polynomials with coefficients from $\mathbb{F}$. For a polynomial
$p\in\mathbb{F}[x]$, $p(x)=\sum_{k=0}^{n}a_{k}x^{k}$, define the function
$f\colon \mathbb{F}[x]\rightarrow\mathbb{F}[x]$ as
\[
f(p)=p',
\]
where $p'(x)=\sum_{k=1}^{n}ka_{k}x^{k-1}$ is the derivative of the polynomial $p$.
Then the function $f$ clearly fulfills
\[
f(p+q)=f(p)+f(q)
\]
and
\[
f(pq)=pf(q)+qf(p)
\]
for all $p, q\in\mathbb{F}[x]$. Hence $f$ is a derivation.
\end{ex}

\begin{ex}
 Let $\mathbb{F}$ be a field, and suppose that we are given a derivation 
 $f\colon \mathbb{F}\to \mathbb{F}$. We define the mapping $f_{0}\colon \mathbb{F}[x]\to \mathbb{F}[x]$ in the following way. 
If $p\in \mathbb{F}[x]$ has the form 
\[
 p(x)=\sum_{k=0}^{n}a_{k}x^{k}, 
\]
then let 
\[
 f_{0}(p)= p^{f}(x)=\sum_{k=0}^{n}f(a_{k})x^{k}. 
\]
Then $f_{0}\colon \mathbb{F}[x]\to \mathbb{F}[x]$ is a derivation. 
\end{ex}

The following lemma says that the above two examples have rather fundamental importance. 

\begin{lem}
 Let $(\mathbb{K}, +, \cdot)$ be a field and let $(\mathbb{F}, +, \cdot)$ be a subfield of $\mathbb{K}$. 
 If $f\colon \mathbb{F}\to \mathbb{K}$ is a derivation, then for any $a\in \mathbb{F}$ and for arbitrary 
 polynomial $p\in \mathbb{F}[x]$ we have 
 \[
  f(p(a))=  p^{f}(a)+f(a)p'(a). 
 \]
\end{lem}

As the following theorem shows, in fields with characteristic zero, with the aid of the notion of algebraic base, 
we are able to construct non-identically zero derivations.

\begin{thm}
Let $\mathbb{K}$ be a field of characteristic zero, let $\mathbb{F}$
be a subfield of $\mathbb{K}$, let $S$ be an algebraic base of $\mathbb{K}$ over $\mathbb{F}$,
if it exists, and let $S=\emptyset$ otherwise.
Let $f\colon \mathbb{F}\to \mathbb{K}$ be a derivation.
Then, for every function $u\colon S\to \mathbb{K}$,
there exists a unique derivation $g\colon \mathbb{K}\to \mathbb{K}$
such that $g \vert_{\mathbb{F}}=f$ and $g \vert_{S}=u$.
\end{thm}

\begin{rem}\label{Rem1}
If $R$ is a commutative ring and $f\colon R\to R$ is a derivation, then  
\[
 f(x^{2})=2xf(x)
 \quad 
 \left(x\in R\right)
\]
is a direct consequence of the Leibniz rule. In general this identity does not characterize derivations among additive functions as the following argument shows. Substituting $z=x+y$ in place of $x$  
\[
 f(z^{2})=2(x+y)f(x+y)= 
 2\left(xf(x)+xf(y)+yf(x)+yf(y)\right). 
\]
On the other hand 
\[
 f(z^{2})=f(x^{2}+2xy+y^{2})= 2xf(x)+2f(xy)+2yf(y). 
\]
Therefore 
\[
 2f(xy)=2xf(y)+2yf(x) 
 \qquad 
 \left(x\in R\right), 
\]
i.e. $2f$ is a derivation. Unfortunately the division by $2$ is not allowed without any further assumption on the ring $R$. 
Therefore some additional restrictions on the ring appear typically in the results. 
\end{rem}

The notion of derivation can be extended in several ways. We will employ the concept of higher order derivations according to Reich \cite{Rei98} and Unger--Reich \cite{UngRei98}. 

\begin{dfn}
 Let $R$ be a ring. The identically zero map is the only \emph{derivation of order zero}. For each $n\in \mathbb{N}$, an additive mapping 
 $f\colon R\to R$ is termed to be a \emph{derivation of order $n$}, if there exists $B\colon R\times R\to R$ such that 
 $B$ is a bi-derivation of order $n-1$ (that is, $B$ is a derivation of order $n-1$ in each variable) and 
 \[
  f(xy)-xf(y)-f(x)y=B(x, y) 
  \qquad 
  \left(x, y\in R\right). 
 \]
 The set of derivations of order $n$ of the ring $R$ will be denoted by $\mathscr{D}_{n}(R)$. 
\end{dfn}

\begin{rem}
\label{pathologic}
Since $\mathscr{D}_{0}(R)=\left\{0\right\}$, the only bi-derivation of order zero is the identically zero function, thus $f\in \mathscr{D}_{1}(R)$ if and only if
  \[
   f(xy)=xf(y)+f(x)y
  \]
that is, the notions of first order derivations and derivations coincide. 
\end{rem}

\begin{rem}
 Let $R$ be a commutative ring and $d\colon R \to R$ be a derivation. 
 Let further $n\in \mathbb{N}$ be arbitrary and 
 \[
  d^{0}= \mathrm{id}, 
  \quad 
  \text{and}
  \quad 
  d^{n}= d\circ d^{n-1} 
  \quad 
  \left(n\in \mathbb{N}\right). 
 \]
Then the mapping $d^{n}\colon R\to R$ is a derivation of order $n$. 

To see this, we will use the following formula 
\[
 d^{k}(xy)= \sum_{i=0}^{k}\binom{k}{i}d^{i}(x)d^{k-i}(y) 
 \qquad
 \left(x, y\in R\right), 
\]
which is valid for any $k\in \mathbb{N}$, and can be proved by induction on $k$. 

For $n=1$ the above statement automatically holds, since the notion of first order derivations and that of derivations coincide. 

Let us assume that there exists an $n\in \mathbb{N}$ so that the statement is fulfilled for any $k\leq n$, that is, 
$d^{k}\in \mathscr{D}_{k}(R)$ holds.  
Then 
\[
 d^{n+1}(xy)=\sum_{i=0}^{n+1}\binom{n+1}{i}d^{i}(x)d^{n+1-i}(y) 
 \qquad 
 \left(x, y\in R\right), 
\]
yielding that 
\[
 d^{n+1}(xy)-xd^{n+1}(y)-yd^{n+1}(x)= \sum_{i=1}^{n}\binom{n+1}{i}d^{i}(x)d^{n+1-i}(y)
 \qquad 
 \left(x, y\in R\right).  
\]
Thus, the only thing that has to be clarified is that the mapping $B\colon R\times R\to R$ defined by 
\[
 B(x, y)= \sum_{i=1}^{n}\binom{n+1}{i}d^{i}(x)d^{n+1-i}(y) 
 \qquad 
 \left(x, y\in R\right)
\]
is a bi-derivation of order $n$. 
Due to the induction hypothesis, for any $k\leq n$, the function $d^{k}$ is a derivation of order $k$. 
Further, we also have $\mathscr{D}_{k-1}(R)\subset \mathscr{D}_{k}(R)$ for all $k\in \mathbb{N}$ and due to the fact that 
$\mathscr{D}_{k}(R)$ is an $R$-module, we obtain that the function $B$ is a derivation of order $n$ in each of its variables.

Of course, $\mathscr{D}_{n}(R)\setminus \mathscr{D}_{n-1}(R)\neq\emptyset$ does not hold in general. To see this, it is 
enough to bear in mind that $\mathscr{D}_{n}(\mathbb{Z})=\left\{ 0\right\}$ for any $n\in \mathbb{N}$.

At the same time, in case $R$ is an integral domain with $\mathrm{char}(R)>n$ or $\mathrm{char}(R)=0$ 
and there is a non-identically zero derivation $d\colon R\to R$, then 
for any $n\in \mathbb{N}$ we have 
\[
 d^{n}\in \mathscr{D}_{n}(R) 
 \qquad 
 \text{and}
 \qquad 
 d^{n}\notin \mathscr{D}_{n-1}(R), 
\]
yielding immediately that $\mathscr{D}_{n}(R)\setminus \mathscr{D}_{n-1}(R)\neq\emptyset$. 

This can also be proved by induction on $n$. For $n=1$ this is automatically true, since $d\colon R\to R$ is a non-identically zero 
derivation. 

Assume now that there is an $n\in \mathbb{N}$ so that the statement holds for any $k\leq n-1$ and suppose to the contrary that 
$d^{n}\in \mathscr{D}_{n-1}(R)$. 

Then due to the definition of higher order derivations, the mapping $B\colon R\times R\to R$ defined by 
\[
B(x, y)= \sum_{i=1}^{n-1}\binom{n}{i} d^{i}(x)d^{n-i}(y)
\qquad 
\left(x, y\in R\right)
\]
has to be a bi-derivation of order $n-2$, i.e., 
\[
 R \ni x \longmapsto B(x, y^{\ast})= \sum_{i=1}^{n-1}\binom{n}{i} d^{i}(x)d^{n-i}(y^{\ast})
\]
has to be a derivation of order $n-2$, where $y^{\ast}\in R$ is kept fixed. 

Clearly, $d^{i}\in \mathscr{D}_{n-2}(R)$ holds, if $i\leq n-2$. 
Since $\mathscr{D}_{n-2}(R)$ is an $R$-module, from this $nd\left(y^{\ast}\right)d^{n-1}\in \mathscr{D}_{n-2}(R)$ would follow, which is a contradiction. 
\end{rem}

\subsection*{Multiadditive functions}

Concerning multiadditive functions we follow the terminology and notations of L. Szé\-kely\-hi\-di \cite{Sze91}, \cite{Sze06}.\index{Székelyhidi, L.}

\begin{dfn}
 Let $G, S$ be commutative semigroups, $n\in \mathbb{N}$ and let $A\colon G^{n}\to S$ be a function. 
 We say that $A$ is \emph{$n$-additive} if it is a homomorphism of $G$ into $S$ in each variable. 
 If $n=1$ or $n=2$ then the function $A$ is simply termed to be \emph{additive}\index{additive function} 
 or \emph{biadditive}\index{biadditive function}, respectively. 
\end{dfn}

The \emph{diagonalization}\index{multiadditive function!--- diagonalization} or \emph{trace} of an $n$-additive function $A\colon G^{n}\to S$ is defined as 
 \[
  A^{\ast}(x)=A\left(x, \ldots, x\right) 
  \qquad 
  \left(x\in G\right). 
 \]
As a direct consequence of the definition each $n$-additive function $A\colon G^{n}\to S$ satisfies 
\begin{multline*}
 A(x_{1}, \ldots, x_{i-1}, kx_{i}, x_{i+1}, \ldots, x_n)
 \\
 =
 kA(x_{1}, \ldots, x_{i-1}, x_{i}, x_{i+1}, \ldots, x_{n})
 \qquad 
 \left(x_{1}, \ldots, x_{n}\in G\right)
\end{multline*}
for all $i=1, \ldots, n$, where $k\in \mathbb{N}$ is arbitrary. 
The same identity holds for any $k\in \mathbb{Z}$ provided that $G$ and $S$ are groups, and 
for $k\in \mathbb{Q}$, provided that $G$ and $S$ are linear spaces over the rationals. 
For the diagonalization of $A$ we have 
\[
 A^{\ast}(kx)=k^{n}A^{\ast}(x)
 \qquad
 \left(x\in G\right). 
\]

One of the most important theoretical results concerning multiadditive functions is the so-called \emph{Polarization formula}, that briefly expresses that every $n$-additive symmetric function is \emph{uniquely} determined by its diagonalization under some conditions on the domain as well as on the range. Suppose that $G$ is a commutative semigroup and $S$ is a commutative group. The action of the {\emph{difference operator}} $\Delta$ on a function  $f\colon G\to S$ is defined by the formula
\[\Delta_y f(x)=f(x+y)-f(x);\]
note that the addition in the argument of the function is the operation of the semigroup $G$ and the subtraction means the inverse of the operation of the group $S$. 

\begin{thm}[Polarization formula]\index{polarization formula}
 Suppose that $G$ is a commutative semigroup, $S$ is a commutative group, $n\in \mathbb{N}$ and $n\geq 1$.  
 If $A\colon G^{n}\to S$ is a symmetric, $n$-additive function, then for all 
 $x, y_{1}, \ldots, y_{m}\in G$ we have 
 \[
  \Delta_{y_{1}, \ldots, y_{m}}A^{\ast}(x)=
  \left\{
  \begin{array}{rcl}
   0 & \text{ if} & m>n \\
   n!A(y_{1}, \ldots, y_{m}) & \text{ if}& m=n.
  \end{array}
  \right.
 \]

\end{thm}

\begin{cor}
 Suppose that $G$ is a commutative semigroup, $S$ is a commutative group, $n\in \mathbb{N}$ and $n\geq 1$.  
 If $A\colon G^{n}\to S$ is a symmetric, $n$-additive function, then for all $x, y\in G$
 \[
  \Delta^{n}_{y}A^{\ast}(x)=n!A^{\ast}(y).
\]
\end{cor}

\begin{lem}
\label{mainfact}
  Let $n\in \mathbb{N}$, $n\geq 1$ and suppose that the multiplication by $n!$ is surjective in the commutative semigroup $G$ or injective in the commutative group $S$. Then for any symmetric, $n$-additive function $A\colon G^{n}\to S$, $A^{\ast}\equiv 0$ implies that 
  $A$ is identically zero, as well. 
\end{lem}

The polarization formula plays the central role in the investigations of functional equations characterizing higher order derivations on a ring. 
This is another reason (see also Remark \ref{Rem1}) 
why some additional restrictions on the the ring appear in the results. 
They essentially correspond to the conditions for the multiplication by $n!$ in the domain as well as in the range in Lemma \ref{mainfact}. 

\section{Multivariate characterizations of higher order derivations}

In what follows we frequently use summation with respect to the cardinality of $I\subset \{1, \ldots, n+1\}$ as $I$ runs through the elements of the power set $2^{\{1, \ldots, n+1\}}$, where $n\in \mathbb{N}$. 
As another technical notation, we introduce the hat operator $\ \widehat{\textrm{\phantom{v}}}\ $ to delete arguments from multivariate expressions. 
Let $R$ be a commutative ring and consider the action of a second order derivation $A\in \mathscr{D}_{2}(R)$ 
on the product of three independent variables as a motivation of the forthcoming results:
 \begin{multline}
  A(x_{1}x_2x_3)-x_1A(x_2x_3)-A(x_1)x_2x_3=B(x_1, x_2x_3)=x_2B(x_1, x_3)+B(x_1, x_2)x_3
	\\
	=x_2\left(A(x_1x_3)-x_1A(x_3)-x_3A(x_1)\right)+\left(A(x_1x_2)-x_1A(x_2)-A(x_1)x_2 \right)x_3.
 \end{multline}

In general 
 \begin{equation}\label{higher}
  \sum_{i=0}^{n}(-1)^i\sum_{\mathrm{card}(I)=i}\left(\prod_{j\in I}x_{j}\right)\cdot 
 A\left(\prod_{k\in \left\{1, \ldots, n+1\right\}\setminus I}x_{k}\right)=0 \qquad (x_{1}, \ldots, x_{n+1}\in R)
 \end{equation}
as a simple inductive argument shows. Conversely suppose that equation \eqref{higher} is satisfied and let us define the (symmetric) biadditive mapping by
\[
B(x,y)=A(xy)-A(x)y-xA(y).
\]
An easy direct computation shows that $A$ satisfies equation \eqref{higher} with $n\in \mathbb{N}$ 
if and only if $B$ satisfies equation \eqref{higher} for each variable  with $n-1\in \mathbb{N}$, where $n\geq 1$. 
By a simple inductive argument we can formulate the following result.  

\begin{thm}\label{Phighmulti}
 Let $A\colon R\to R$ be an additive mapping, where $R$ is a commutative ring,  $n\in \mathbb{N}$ and $n\geq 1$.   $A\in \mathscr{D}_{n}(R)$ if and only if 
 \begin{equation}
  \sum_{i=0}^{n}(-1)^i\sum_{\mathrm{card}(I)=i}\left(\prod_{j\in I}x_{j}\right)\cdot 
 A\left(\prod_{k\in \left\{1, \ldots, n+1\right\}\setminus I}x_{k}\right)=0 \qquad (x_{1}, \ldots, x_{n+1}\in R).
 \end{equation}
\end{thm}

As a generalization of the previous result we admit more general coefficients in equation (\ref{higher}). Some additional requirements for the ring $R$ should be also formulated.

\begin{dfn} A commutative unitary ring $R$ is called {\emph{linear}} if it is a linear space over the field of rationals.
\end{dfn}

\begin{rem}
\label{ringconditions}
It can be easily seen that a linear ring admits the multiplication 
of the ring elements by fractions such as $1/2, \ldots, 1/n, \ldots$. 
Using the linear space properties the field of the rationals $\mathbb{Q}$ can be isomorphically embedded into $R$ as a subring. Therefore each element  
\[
\mathbf{1},\ \mathbf{2}=\mathbf{1}+\mathbf{1}, \ldots, \mathbf{n}=\underbrace{\mathbf{1}+\ldots+\mathbf{1}}_{\text{$n$-times}}
\]
is invertible, where $\mathbf{1}$ is the unity of the ring, $n\in \mathbb{N}$ and $n\geq 1$. 
Hence the multiplication by $n!$ is injective in $R$ as the domain of the higher order derivations (cf.~Lemma \ref{mainfact}). 
Some typical examples for linear rings: fields, rings formed by matrices over a field, polynomial rings over a field. 
Another natural candidates are the integral domains. 
Since each cancellative semigroup $G$ can be isomorphically embedded into a group we have that each integration domain $R$ 
can be isomorphically embedded into a field $\mathbb{F}$ as a subring. In the sense of Definition \ref{D2.1.1} we can take $\mathbb{F}$ 
as the range of derivations on the subring $R$. 
Such an extension of $R$ provides the multiplication by $n!$ to be obviously surjective in the range $\mathbb{F}$, 
assuming that the characteristic is zero (cf.~Lemma \ref{mainfact}). 
\end{rem}

In what follows we do not use special notation for the unity of the ring. The meaning of $1$ depends on the context as usual. 

\begin{thm}
\label{Phighmulti1}
 Let $A\colon R\to R\subset \mathbb{F}$ be an additive mapping satisfying $A(1)=0$, where $R$ is an integral domain with unity, 
 $\mathbb{F}$ is its embedding field, assuming that $\mathrm{char}(\mathbb{F})=0$ and  $n\in \mathbb{N}$. 
 $A\in \mathscr{D}_{n}(R)$ if and only if there exist constants $a_{1}, \ldots, a_{n+1}\in \mathbb{F}$, not all zero, such that 
 \begin{equation}
\label{gen_singlemultivariate}
\sum_{i=0}^{n}\dfrac{a_{n+1-i}}{\binom{n+1}{i}}\sum_{\mathrm{card}(I)=i}\left(\prod_{j\in I}x_{j}\right)\cdot 
 A\left(\prod_{k\in \left\{1, \ldots, n+1\right\}\setminus I}x_{k}\right)=0,
\end{equation}
where $x_{1}, \ldots, x_{n+1}\in R$.  

Especially, $\displaystyle\sum_{i=1}^{n+1}ia_{i}=0$ provided that $A$ is a not identically zero solution.  
\end{thm}
\begin{proof}
If $A\in \mathscr{D}_{n}(R)$ then we have equation \eqref{gen_singlemultivariate} with 
\[
 a_{n+1-i}=(-1)^{i}\binom{n+1}{i} 
 \qquad 
 \left(i=0, \ldots, n\right); 
\]
none of the constants $a_{1}, \ldots, a_{n+1}$ is zero and $\displaystyle\sum_{i=1}^{n+1}ia_{i}=0$. 
Conversely, let $n\in \mathbb{N}$, $n\geq 1$ and $A\colon R\to R\subset \mathbb{F}$ be an additive function such that $A(1)=0$. 
Suppose that equation \eqref{gen_singlemultivariate} holds for all $x_{1}, \ldots, x_{n+1}\in R$. Substituting 
$$x_1=\ldots =x_n=1\ \ \textrm{and}\ \ x_{n+1}=x\in R$$
it follows that 
\[
 \left(\sum_{i=1}^{n+1}ia_{i}\right)\cdot A(x)=0
 \qquad 
 \left(x\in R\right),
\]
i.e. $\displaystyle \sum_{i=1}^{n+1}ia_{i}=0$ provided that $A$ is a not identically zero solution of equation (\ref{gen_singlemultivariate}). If $n=1$, then it takes the form  
\[
 2 a_{2}A(xy)+a_{1}\left(xA(y)+yA(x)\right)=0 
 \qquad 
 \left(x, y\in R\right). 
\]
Since $2a_{2}+a_{1}=0$, it follows that $a_2\neq 0$. Otherwise $0=a_2=a_1$ which is a contradiction. Therefore  
\[
 a_{2}A(xy)-a_{2}\left(xA(y)+yA(x)\right)=0 
 \qquad 
 \left(x, y\in R\right) 
\]
and 
\[
 A(xy)=xA(y)+yA(x)
 \qquad 
 \left(x, y\in R\right).
\]
This means that $A\in \mathscr{D}_{1}(R)$, i.e. the statement holds for $n=1$. The inductive argument can be completed as follows. Taking $x_{n+1}=1$ equation (\ref{gen_singlemultivariate}) gives that
 \begin{multline*}
0=\sum_{i=1}^{n} \dfrac{a_{n+1-i}}{\binom{n+1}{i}} \sum_{n+1\in I}\sum_{\mathrm{card}(I)=i} \left(\prod_{j\in I\setminus \{n+1\}}x_{j}\right)\cdot 
 A\left(\prod_{k\in \left\{1, \ldots, n\right\}\setminus (I\setminus \{n+1\})}x_{k}\right)
\\
+\sum_{i=0}^{n-1} \dfrac{a_{n+1-i}}{\binom{n+1}{i}} \sum_{n+1 \notin I}\sum_{\mathrm{card}(I)=i}\left(\prod_{j\in I}x_{j}\right)\cdot 
 A\left(\prod_{k\in \left\{1, \ldots, n\right\}\setminus I}x_{k}\right)+\dfrac{a_{1}}{\binom{n+1}{n}}\left(\prod_{j\in \{1, \ldots, n\}}x_{j}\right) \cdot A(1)
\\
=\sum_{i=0}^{n-1}\left( \dfrac{a_{n+1-(i+1)}}{\binom{n+1}{i+1}}+ \dfrac{a_{n+1-i}}{\binom{n+1}{i}}\right) \sum_{\mathrm{card}(I)=i}\left(\prod_{j\in I}x_{j}\right)\cdot 
 A\left(\prod_{k\in \left\{1, \ldots, n\right\}\setminus I}x_{k}\right)
 \\
 =\sum_{i=0}^{n-1} \dfrac{\tilde{a}_{(n-1)+1-i}}{\binom{n-1+1}{i}}
\sum_{\mathrm{card}(I)=i}\left(\prod_{j\in I}x_{j}\right)\cdot 
 A\left(\prod_{k\in \left\{1, \ldots, (n-1)+1\right\}\setminus I}x_{k}\right) \qquad (x_1, \ldots, x_n\in R) 
\end{multline*}
because $A(1)=0$. We have that $A\in \mathscr{D}_{n-1}(R) \subset \mathscr{D}_{n}(R)$ (inductive hypothesis) or all the coefficients are zero, 
that is, 
\[
 \tilde{a}_{(n-1)+1-i}=\binom{n-1+1}{i}\left(\dfrac{a_{n+1-(i+1)} }{\binom{n+1}{i+1}}+\dfrac{a_{n+1-i}}{\binom{n+1}{i}} \right)=0 
 \qquad 
 \left(0\leq i\leq n-1\right)
\]
and, consequently,  
\[
 (n+1-i)a_{n+1-i}+(i+1)a_{n+1-(i+1)}=0
 \qquad 
 \left(0\leq i\leq n-1\right).
\]
Therefore 
\begin{equation}
\label{recursion}
 a_{n+1-(i+1)}=(-1)^{i+1}\binom{n+1}{i+1}a_{n+1}
 \qquad 
 \left(0\leq i\leq n-1\right). 
\end{equation}
Substituting into \eqref{gen_singlemultivariate} 
\[
 a_{n+1}\cdot \sum_{i=0}^{n}(-1)^{i}\sum_{\mathrm{card}(I)=i} \left(\prod_{j\in I}x_{j}\right)\cdot 
 A\left(\prod_{k\in \left\{1, \ldots, n+1\right\}\setminus I}x_{k}\right)=0 
 \qquad 
 \left(x\in R\right).
\]
This means that $A\in \mathscr{D}_{n}(R)$ since $a_{n+1}$ can not be zero due to the recursive formula \eqref{recursion}
and the condition to provide the existence of a nonzero element among the constants $a_1, \ldots, a_{n+1}$.
\end{proof}

\begin{rem}\label{rem1}
 The condition for $R$ to be an integral domain with unity embeddable into a field $\mathbb{F}$ with characteristic zero, 
 is used only in the last step of the 
 proof because of the  division by $a_{n+1}$. If the coefficients are supposed to be rationals, then the 
 previous statement holds for a linear commutative ring with unity, see the next theorem. 
 \end{rem}

\begin{thm}\label{Thm_mainmultivariate}
 Let $f_{1}, \ldots, f_{n+1}\colon\allowbreak R\to R$ be additive functions such that $f_{i}(1)=0$ for all $i=1, \ldots, n+1$, where $R$ is a linear commutative ring with unity, $n\in \mathbb{N}$ and $n\geq 1$. Functional equation 
 \begin{equation}\label{Eq7multivariate}
\displaystyle\sum_{i=0}^{n}\dfrac{1}{\binom{n+1}{i}}\sum_{\mathrm{card}(I)=i}\left(\prod_{j\in I}x_{j}\right)\cdot f_{n+1-i}\left(\prod_{k\in \left\{1, \ldots, n+1\right\}\setminus I}x_{k}\right)
=0
 \;
  \left(x_1, \ldots, x_{n+1}\in R\right)
 \end{equation}
holds if and only if 
\[
 f_{n+1-i}=(-1)^{i}\sum_{k=0}^{i}\binom{n+1-i+k}{k}D_{n-i+k} 
 \qquad 
 \left(i=0, \ldots, n\right), 
\]
where for all possible indices $i$, we have $D_{i}\in \mathscr{D}_{i}(R)$. 
\end{thm}
\begin{proof}
If $n=1$ then the equation takes the form
\[
2f_2(x_1x_2)+x_1f_1(x_2)+x_2f_{1}(x_1)=0.
\]
Substituting $x_1=x$ and $x_2=1$ we have that
\[
2f_2(x)+f_1(x)=0
\]
and, consequently,
\[
f_2(x_1x_2)-x_1f_2(x_2)-x_2f_{2}(x_1)=0.
\]
This means that $f_2=D_1\in \mathscr{D}_{1}(R)$ and $f_1=-D_0-2D_1$, where $D_0\in \mathscr{D}_{0}(R)$ is the identically zero function. The converse of the statement is clear under the choice $f_2=D_1$ and $f_1=-D_0-2D_1$. The inductive argument is the same as in the proof of the previous theorem by the formal identification $f_{n+1-i}=a_{n+1-i}A.$ We have
 \begin{equation}
0=\sum_{i=0}^{n-1}\dfrac{1}{\binom{n-1+1}{i}}\sum_{\mathrm{card}(I)=i}\left(\prod_{j\in I}x_{j}\right)\cdot \tilde{f}_{(n-1)+1-i}\left(\prod_{k\in \left\{1, \ldots, n\right\}\setminus I}x_{k}\right),
\end{equation}
where
$$\tilde{f}_{(n-1)+1-i}=(i+1)f_{n+1-(i+1)}+ (n+1-i)f_{n+1-i} \qquad (0\leq i \leq n-1).$$
Using the inductive hypothesis 
\begin{equation}
 (n+1-i)f_{n+1-i}+(i+1)f_{n+1-(i+1)}=(-1)^{i}\sum_{k=0}^{i}\binom{n-i+k}{k}\widetilde{D}_{(n-1)-i+k},
\end{equation}
where $\widetilde{D_{i}}\in \mathscr{D}_{i}(R)$ for all possible indices. Especially, all the unknown functions $f_1, \ldots, f_{n}$ can be 
expressed in terms of $f_{n+1}$ and $\widetilde{D}_{0}, \ldots, \widetilde{D}_{n-1}$ in a recursive way:
\begin{equation}\label{Eq8multivariateversion1}
 f_{n+1-(i+1)}= (-1)^{i}\left\{\left(\sum_{k=0}^{i}\binom{n-k}{i-k}\dfrac{\widetilde{D}_{n-(k+1)}}{k+1}\right)-\binom{n+1}{i+1}f_{n+1}\right\},
\end{equation}
where $i=0, \ldots, n-1$. Rescaling the indices, 
\begin{multline}\label{Eq8multivariate}
 f_{n+1-i}= f_{n+1-(i-1+1)}
 \\
 =(-1)^{i-1}\left\{\left(\sum_{k=0}^{i-1}\binom{n-k}{i-1-k}\dfrac{\widetilde{D}_{n-(k+1)}}{k+1}\right)-\binom{n+1}{i}f_{n+1}\right\}, 
\end{multline}
where $i=1, \ldots, n$. Substituting into equation (\ref{Eq7multivariate}) 
\begin{equation*}
\displaystyle\sum_{i=0}^{n}(-1)^i\sum_{\mathrm{card}(I)=i}\left(\prod_{j\in I}x_{j}\right)\cdot f_{n+1}\left(\prod_{k\in \left\{1, \ldots, n+1\right\}\setminus I}x_{k}\right)=0,
\end{equation*}
because the term containing $\widetilde{D}_{l}$ is of the form 
\begin{multline*}
 \displaystyle\sum_{i=n-l}^{n}\frac{(-1)^{i-1}}{\binom{n+1}{i}} \binom{l+1}{i-n+l}\sum_{\mathrm{card}(I)=i}\left(\prod_{j\in I}x_{j}\right)
\cdot \dfrac{\widetilde{D}_{l}}{n-l}\left(\prod_{k\in \left\{1, \ldots, n+1\right\}\setminus I}x_{k}\right)
\\
=
\dfrac{1}{\binom{n+1}{n-l}}\displaystyle\sum_{i=n-l}^{n}(-1)^{i-1}\binom{i}{n-l}\sum_{\mathrm{card}(I)=i}\left(\prod_{j\in I}x_{j}\right)
\cdot \dfrac{\widetilde{D}_{l}}{n-l}\left(\prod_{k\in \left\{1, \ldots, n+1\right\}\setminus I}x_{k}\right)
\\
=
\dfrac{(-1)^{n-l-1}}{\binom{n+1}{n-l}}\displaystyle\sum_{m=0}^{l}(-1)^{m}\binom{m+n-l}{n-l}
\\
\cdot\sum_{\mathrm{card}(I)=n-l+m}\left(\prod_{j\in I}x_{j}\right)\cdot \dfrac{\widetilde{D}_{l}}{n-l}\left(\prod_{k\in \left\{1, \ldots, n+1\right\}\setminus I}x_{k}\right)
\\
=\dfrac{(-1)^{n-l-1}}{\binom{n+1}{n-l}}
\displaystyle\sum_{1\leq j_1 < \ldots < j_{l+1}\leq n+1}\left(\prod_{k\in \left\{ 1, \ldots, n+1\right\}\setminus \left\{j_{1}, \ldots, j_{l+1}\right\}}x_{k}\right)
\\
\cdot \displaystyle\sum_{m=0}^{l}(-1)^m\sum_{\mathrm{card}(J)=m}
\left(\prod_{j\in J}x_{j}\right)\cdot \dfrac{\widetilde{D}_{l}}{n-l}\left(\prod_{k\in \left\{j_1, \ldots, j_{l+1}\right\}\setminus J}x_{k}\right)=0, 
\end{multline*}
due to Theorem \ref{Phighmulti}. 

Finally, let 
\[
 D_{n}= f_{n+1} 
 \quad 
 \text{and}
 \quad 
 D_{n-(k+1)}=-\dfrac{\widetilde{D}_{n-(k+1)}}{k+1}
 \qquad
 \left(k=0, \ldots, n-1\right). 
\]
Taking $m=i-1-k$ in formula \eqref{Eq8multivariate} it follows that
\begin{multline*}
f_{n+1-i}=(-1)^{i-1}\left\{\left(\sum_{m=0}^{i-1}\binom{n+1-i+m}{m}\dfrac{\widetilde{D}_{n-i+m}}{i-m}\right)-\binom{n+1}{i}f_{n+1}\right\}
\\
=(-1)^{i}\left\{\left(\sum_{m=0}^{i-1}\binom{n+1-i+m}{m}D_{n-i+m}\right)+\binom{n+1}{i}f_{n+1}\right\}
\\
=(-1)^{i}\sum_{m=0}^{i}\binom{n+1-i+m}{m}D_{n-i+m}
\end{multline*}
as had to be proved. 
The converse statement is a straightforward calculation. 
\end{proof}

\begin{rem}
Theorem \ref{Thm_mainmultivariate} can be also stated in case of integral domains provided that the range of the functions 
$f_1, \ldots, f_{n+1}$ is extended to a field of characteristic zero, containing $R$ as a subring; see Remark \ref{ringconditions}. 
The proof works without any essential modification. 
\end{rem}

\section{Univariate characterizations of higher order derivations}

Each multivariate characterization implies an univariate characterization of 
the higher order derivations under some mild conditions on the ring $R$ due to Lemma \ref{mainfact}. 
The ring $R$ is supposed to be a linear commutative, unitary ring or an integral domain embeddable to a field of characteristic zero; 
see also Remark \ref{ringconditions}. 
The following results can be found in \cite{Eba15}, \cite{Eba17} but the proofs are essentially different. 
We present them as direct consequences of the multivariate characterizations by taking the diagonalization of the formulas. 
This provides unified arguments of specific results but we can also use the idea as a 
general method to solve functional equations of the form \eqref{Ebanks}; see Examples I and Examples II. 

\begin{cor}\label{Prop3}
 Let $A\colon R\to R$ be an additive mapping, where $R$ is a linear commutative,
unitary ring or an integral domain with unity, embeddable to a field
 of characteristic zero and 
 $n\in \mathbb{N}$ and $n\geq 1$. Then $A\in \mathscr{D}_{n}(R)$ if and only if 
 \begin{equation}\label{high_single}
  \sum_{i=0}^{n}(-1)^{i}\binom{n+1}{i}x^{i}A\left(x^{n+1-i}\right)=0 
  \qquad 
  \left(x\in R\right). 
 \end{equation}
\end{cor}

\begin{cor}\label{cor3}
Let $A\colon R\to R\subset \mathbb{F}$ be an additive mapping satisfying $A(1)=0$, where $R$ is an integral domain with unity, 
$\mathbb{F}$ is its embedding field with $\mathrm{char}\left(\mathbb{F}\right)=0$,  $n\in \mathbb{N}$ and $n\geq 1$. 
$A\in \mathscr{D}_{n}(R)$ if and only if there exist constants $a_{1}, \ldots, a_{n+1}\in F$, not all zero, such that 
 \begin{equation}\label{gen_single}
  \sum_{i=0}^{n}a_{n+1-i}x^{i}A\left(x^{n+1-i}\right)=0
 \end{equation}
 for any $x\in R$. 

Especially, $\displaystyle\sum_{i=1}^{n+1}ia_{i}=0$ provided that $A$ is a not identically zero solution. 
\end{cor}

As an application of Theorem \ref{Thm_mainmultivariate} we are going to give the general form of the solutions of functional equations of the form 
\begin{equation}\label{Eq1}
 \sum_{k=1}^{n}x^{p_{k}}f_{k}(x^{q_{k}})=0
 \qquad 
 \left(x\in R\right),
\end{equation}
where $n\in \mathbb{N}$, $n\geq 1$, $p_{k}, q_{k}$ are given nonnegative integers for all $k=1, \ldots, n$ and $f_{1}, \ldots, f_{n}\colon R\to R$ are additive functions on the ring $R$. 
The problem is motivated by B.~Ebanks' paper \cite{Eba15}, \cite{Eba17} although there are some earlier relevant results. 
For example the case $n=2$ is an easy consequence of \cite[Theorem 6]{Gse13}.

Using the homogeneity of additive functions it can be easily seen that equation \eqref{Eq1} can be assumed to be of constant degree of homogeneity:
$$p_{k}+q_{k}=l \qquad (k=1, \ldots, n);$$
see \cite[Lemma 2.2]{Eba15}. In other words collecting the addends of equation \eqref{Eq1} 
of the same degree of homogeneity we discuss the vanishing of the left hand side term by term. 
To formulate the multivariate version of the functional equation let us define the mapping  
\[
 \Phi(x_{1}, \ldots, x_{l})=
 \sum_{k=1}^{n}
 \dfrac{1}{\binom{l}{p_{k}}} \sum_{\mathrm{card}(I)=p_{k}} \left(\prod_{j\in \left\{1, \ldots, l\right\}\setminus I}x_{j}\right)\cdot  f_{k}\left(\prod_{i\in I}x_{i}\right)
 \quad 
 \left(x_{1}, \ldots, x_l\in R\right), 
\]
where the summation is taken for all subsets $I$ of cardinality $p_{k}$ of the index set 
$\left\{1, 2,\ldots, l\right\}$. Due to the additivity of the functions $f_{1}, \ldots, f_{n}$, the mapping $\Phi\colon R^{l}\to R$ 
is a symmetric $l$-additive function with vanishing trace  
\[
 \Phi^{\ast}(x)=\Phi(x, \ldots, x)=\sum_{k=1}^{n}x^{p_{k}}f_{k}(x^{q_{k}}) =0
 \qquad 
 \left(x\in R\right)
\]
and we can conclude the equivalence of the following statements by Lemma \ref{mainfact}: 

\begin{enumerate}[(i)]
 \item The functions $f_{1}, \ldots, f_{n}$ fulfill equation \eqref{Eq1}. 
 \item For any $x_{1}, \ldots, x_{l}\in R$
 \[
 \sum_{k=1}^{n}
 \dfrac{1}{\binom{l}{p_{k}}} \sum_{\mathrm{card}(I)=p_{k}} \left(\prod_{j\in \left\{1, \ldots, l\right\}\setminus I}x_{j}\right) \cdot f_{k}\left(\prod_{i\in I}x_{i}\right)=0.
 \]
\end{enumerate}

In what follows we summarize all the simplifications we use to solve equation \eqref{Eq1}.

\begin{enumerate}[(i)]
\item[($\mathscr{C}0)$] $R$ is a linear commutative, unitary ring or an integral domain with unity embeddable to a field of characteristic zero. 
\item[($\mathscr{C}1)$] Each addend has the same degree of homogeneity, i.e. $p_{k}+q_{k}=l$ for all $k=1, \ldots, n$.
\end{enumerate}

If $p_{i}=p_{j}$ for some different indices $i\neq j\in \left\{1, \ldots, n\right\}$ then $q_{i}=q_{j}$ by the constancy of the degree of homogeneity and we can write that 
$$x^{p_{i}}f_{i}(x^{q_{i}})+x^{p_{j}}f_{j}(x^{q_{j}})=x^{p_{i}}\widetilde{f}(x^{q_{i}}),$$
where $\widetilde{f}=f_{i}+f_{j}$. Therefore the number of the unknown functions has been reduced. If $q_{i}=0$ then, by choosing $j\neq i$, we can write equation (\ref{Eq1}) into the form 
\[
 \sum_{\nu\in \left\{1, \ldots, n\right\}\setminus \left\{ i, j\right\}}x^{p_{\nu}}f_{\nu}(x^{q_{\nu}})+x^{p_{j}}\widetilde{f}(x^{q_{j}})=0,
\]
where $ \widetilde{f}(x)= f_{i}(1) x+f_{j}(x)$ and the number of the unknown functions has been reduced again. Without loss of the generality we can suppose that  

\begin{enumerate}[(i)]
  \item[($\mathscr{C}2)$] $p_1, \ldots, p_n$ are pairwise different nonnegative integers, 
  $q_1, \ldots, q_n$ are positive, pairwise different integers. 

Especially, $l\geq n$ because of $(\mathscr{C}1)$. By multiplying equation \eqref{Eq1} with $x$ if necessary we have that $l\geq n+1$.
\item[($\mathscr{C}3)$] Moreover, for any $k=1, \ldots, n$, condition $f_{k}(1)=0$ can be assumed. Otherwise, let us introduce the functions 
 \[
  \widetilde{f_{k}}(x)=f_{k}(x)-f_{k}(1) x 
  \qquad 
  \left(x\in R\right). 
 \]
They are obviously additive and, for any $x\in R$
\[
 \sum_{k=1}^{n}x^{p_{k}}\widetilde{f_{k}}(x^{q_{k}})=
 \sum_{k=1}^{n}x^{p_{k}}\left[f_{k}(x^{q_{k}})-f_{k}(1) x^{q_{k}} \right]=
 \sum_{k=1}^{n}x^{p_{k}}f_{k}(x^{q_{k}})-x^{l}\sum_{k=1}^{n}f_{k}(1)=0
\]
because $\displaystyle\sum_{k=1}^{n}f_{k}(1)=0$ due to \eqref{Eq1}.
 \end{enumerate}

Assuming conditions $(\mathscr{C}0)$, $(\mathscr{C}1)$, $(\mathscr{C}2)$ and $(\mathscr{C}3)$ it is enough to investigate functional equation 
\begin{equation}\label{Eq2}
 \sum_{i=0}^{n}x^{i}f_{n+1-i}(x^{n+1-i})=0
 \qquad 
 \left(x\in \mathbb{R}\right).
\end{equation}
For those exponents that do not appear in the corresponding homogeneous term of the original equation we assign the identically zero function (that is clearly additive).

\begin{cor}\label{Thm_main}
 Let $f_{1}, \ldots, f_{n+1}\colon\allowbreak R\to R$ be additive functions such that $f_{i}(1)=0$ for all $i=1, \ldots, n+1$, where $R$ is a linear commutative ring with unity, $n\in \mathbb{N}$ and $n\geq 1$. Functional equation 
 \begin{equation}\label{Eq7}
  \sum_{i=0}^{n}x^{i}f_{n+1-i}\left(x^{n+1-i}\right)=0 
  \qquad 
  \left(x\in R\right)
 \end{equation}
holds, if and only if 
\[
 f_{n+1-i}=(-1)^{i}\sum_{k=0}^{i}\binom{n+1-i+k}{k}D_{n-i+k} 
 \qquad 
 \left(i=0, \ldots, n\right), 
\]
where for all possible indices $i$, we have $D_{i}\in \mathscr{D}_{i}(R)$. 
\end{cor}

\begin{proof}
The statement is a direct consequence of Theorem \ref{Thm_mainmultivariate} and Lemma \ref{mainfact}.
\end{proof} 

\begin{rem}
Corollary \ref{Thm_main} can be also stated in case of integral domains provided that the range of the functions $f_1, \ldots, f_{n+1}$ 
is extended to a field having characteristic zero, 
containing $R$ as a subring; see Remark \ref{ringconditions}. 
The proof is working without any essential modification. 
\end{rem}

\subsection*{Examples I: equations without missing powers}

As an application of the results presented above, we will show some examples. 

\begin{ex}
 Let $f\colon R\to R$ be a non-identically zero additive function and assume that 
 \[
  f(x^{3})+xf(x^{2})-2x^{2}f(x)=0
 \]
is fulfilled for any $x\in R$. Then the function $A(x)=f(x)-f(1) x$ satisfies the conditions of Corollary \ref{cor3}. Therefore there exists a derivation $D_{2}\in \mathscr{D}_{2}(R)$ such that  
\[
 f(x)=D_{2}(x)+f(1) x 
 \qquad
 \left(x\in R\right). 
\]
Moreover, if we define $a_{3}=1$, $a_{2}= 1$ and $a_{1}=-2$, then $3a_{3}+2a_{2}+a_{1}\neq 0$, i.e. $D_{2}\equiv 0$  in the sense of Corollary \ref{cor3}. The solution is $f(x)=f(1)x$, where $f(1)\neq 0$ and $x\in R$. 
\end{ex}

\begin{ex}
  Let $f\colon R\to R$ be an additive function and assume that 
 \[
  f(x^{5})-5xf(x^{4})+10x^{2}f(x^{3})-10x^{3}f(x^{2})+5x^{4}f(x)=0
 \qquad 
 \left(x\in R\right). 
 \]
Let 
\[
 a_{5-i}=(-1)^{i}\binom{5}{i} 
 \qquad 
 \left(i=0, 1, 2, 3, 4\right). 
\]
Then we have 
\[
 \sum_{i=0}^{4}a_{5-i}x^{i}f(x^{5-i})=0 
 \qquad 
 \left(x\in R\right). 
\]
Thus there exists a fourth order derivation $D_{4}\in \mathscr{D}_{4}(R)$ such that 
\[
 f(x)=D_{4}(x)+f(1)x 
 \qquad 
 \left(x\in R\right). 
\]
Since $f(1)=0$, the solution is $f(x)=D_{4}(x)$, where $x\in R.$
\end{ex}

\subsection*{Examples II: equations with missing powers}

The following examples illustrate how to use the method of free variables in the solution of 
functional equations with missing powers instead of the direct application of Theorem \ref{Thm_mainmultivariate}/Corollary \ref{Thm_main}. 
The key step is the substitution of value $1$ as many times as the number of the missing powers 
(due to the symmetry of the variables there is no need to specify the positions for the substitution). 
This provides the reduction of the homogeneity degree of the functional equation to avoid formal (identically zero) terms in the solution. 
It is given in a more effective and subtle way.  
The method obviously shows how the number of the nonzero coefficients is 
related to the maximal order of the solution. 
In \cite{EbaRieSah} it was shown that if the number of nonzero coefficients $a_i \in R$ is $m$, 
then the solution of 
\[
\sum_{i=1}^{n} a_i x^{p_i}f(x^{q_i})=0
\]
is a derivation of order at most $m-1$.

\begin{ex}
Assume that we are given two additive functions 
$f, g\colon R\to R$ such that $f(1)=g(1)=0$ and 
\[
 f(x^{3})+x^{2}g(x)=0
 \qquad 
 \left(x\in R\right). 
\]
If we define the functions $f_{1}, f_{2}, f_{3}\colon R\to R$ as $f_{3}=f$, $f_{2}=0$ and $f_{1}=g$ then, by using Corollary \ref{Thm_main} with $n=2$, it follows that 
\[
 \begin{array}{rcl}
  f_{3}&=&D_{2}\\
  f_{2}&=&-3D_{2}-D_{1}\\
  f_{1}&=&3D_{2}+2D_{1}+D_{0}, 
 \end{array}
\]
where $D_{i}\in \mathscr{D}_{i}(R)$, if $i=0, 1, 2$.
It is a direct application of Corollary \ref{Thm_main} used by B.~Ebanks in \cite{Eba15}, \cite{Eba17}. 
Another way of the solution is to formulate the multivariate version of the equation in the first step: 
\[
 3f(x_{1}x_{2}x_{3})+x_{1}x_{2}g(x_{3})+x_{1}x_{3}g(x_{2})+x_{2}x_{3}g(x_{1})=0
 \qquad
 \left(x_{1}, x_{2}, x_{3}\in R\right). 
\]
If $x_{1}=x_{2}=x$ and $x_{3}=1$, then we get that 
\[
 3f(x^{2})+2xg(x)=0 
 \qquad 
 \left(x\in R\right). 
\]
Applying Corollary \ref{Thm_main} with $n=1$,  $f_{2}=3f$ and $f_{1}=2g$, it follows that  
\[
 \begin{array}{rcl}
  f_{2}&=&D_{1}\\
  2D_{1}+D_{0}+f_{1}&=&0, 
 \end{array}
\]
where $D_{i}\in \mathscr{D}_{i}(R)$ for all $i=0, 1$. Therefore $3f=D_{1}$ and $g=-D_{1}$, where $D_{1}\in \mathscr{D}_{1}(R)$. 
\end{ex}

\begin{ex}
 Let $f, g, h\colon R\to R$ be additive functions so that $f(1)=g(1)=h(1)=0$. Furthermore, assume that 
 \[
  f(x^{5})+xg(x^{4})+x^{4}h(x)=0
 \]
is fulfilled for all $x\in R$. To determine the functions $f, g, h$, we will show two ways. 

The first is a direct application of the results above used by 
B.~Ebanks in \cite{Eba15}, \cite{Eba17}. 
Let us define the functions $f_{1}, f_{2}, f_{3}, f_{4}, f_{5}\colon R\to R$ as $f_{1}= h$, $f_{2}=0$, $f_{3}=0$, $f_{4}=g$ and $f_{5}=f$. 
Then the equation takes the form 
\[
 \sum_{i=0}^{4}x^{i}f_{5-i}(x^{5-i})=0 
 \qquad 
 \left(x\in R\right). 
\]
and, by Corollary \ref{Thm_main}, for any $i=0, 1, 2, 3, 4$
\[
 f_{5-i}= (-1)^{i}\sum_{k=0}^{i}\binom{5-i+k}{k}D_{4-i+k}
\]
holds, that is 
\[
 \begin{array}{rcl}
  f_{5}-D_{4}&=&0\\
  5\,D_{4}+f_{4}+D_{3}&=&0\\
  -10\,D_{4}-4\,D_{3}+f_{3}-D_{2}&=&0\\
  10\,D_{4}+6\,D_{3}+3\,D_{2}+f_{2}+D_{1}&=&0\\
  -5\,D_{4}-4\,D_{3}-3\,D_{2}-2\,D_{1}+f_{1}-D_{0}&=&0, 
   \end{array}
\]
where for all $i=0, 1, 2, 3, 4$ we have $D_{i}\in \mathscr{D}_{i}(R)$. 
Bearing in mind the above notations, for the functions $f, g$ and $h$ this yields that 
\[
\begin{array}{rcl}
 f-D_{4}&=&0\\
  5\,D_{4}+g+D_{3}&=&0\\
  -10\,D_{4}-4\,D_{3}-D_{2}&=&0\\
  10\,D_{4}+6\,D_{3}+3\,D_{2}+D_{1}&=&0\\
  -5\,D_{4}-4\,D_{3}-3\,D_{2}-2\,D_{1}+h&=&0. 
   \end{array} 
\]

The second way gives a much more precise form of the unknown functions by formulating the multivariate version of the equation in the first step: 
\begin{multline*}
5 f(x_1 x_2 x_3 x_4 x_5)
\\
+g(x_2 x_3 x_4 x_5) x_1 + g(x_1 x_3 x_4 x_5) x_2+ g(x_2 x_1 x_4 x_5) x_3 + g(x_2 x_3 x_1 x_5) x_4+ g(x_2 x_3 x_4 x_1) x_5 
\\
+x_2 x_3 x_4 x_5 h(x_1)+ x_1 x_3 x_4 x_5 h(x_2) + x_2 x_1 x_4 x_5 h(x_3)+ x_2 x_3 x_1 x_5 h(x_4) + x_2 x_3 x_4 x_1 h(x_5)
\\
=0,
\end{multline*}
where $x_{1}, x_{2}, x_{3}, x_{4}, x_{5}\in R$. Substituting $x_{1}=x_{2}=x_{3}= x $ and $x_{4}=x_{5}=1$ we get that 
\[
 5f(x^{3})+2g(x^{3})+3xg(x^{2})+3x^{2}h(x)=0. 
\]
Taking $f_{3}= 5f+2g$, $f_{2}= 3g$ and $f_{1}= 3h$ it follows that  
\[
 f_{3}(x^{3})+xf_{2}(x^{2})+x^{2}f_{1}(x)=0 
 \qquad 
 \left(x\in R\right) 
\]
and, by Corollary \ref{Thm_main}, 
\[
 f_{3}= D_{2},
 \quad 
 f_{2}= -D_{1}-3D_{2},
 \quad 
 f_{1}=2D_{1}+3D_{2}, 
\]
where $D_{i}\in \mathscr{D}_{i}(R)$ for all $i=1, 2$. This means that 
\[
 \begin{array}{rcl}
  15 f&=&2D_{1}+9D_{2}\\
  3g &=&-D_{1}-3D_{2}\\
  3h&=&2D_{1}+3D_{2}, 
 \end{array}
\]
where $D_{i}\in \mathscr{D}_{i}(R)$ for all $i=1, 2$.
\end{ex}

\section{The application of the spectral synthesis in the solution of linear functional equations}

In what follows we present another approach to the problem of linear functional equations characterizing derivations among additive mappings in the special case of a finitely generated field $K$ over the field $\mathbb{Q}$ of rationals as the domain $R$ of the equation, i.e. $\mathbb{Q}\subset K=\mathbb{Q}(x_1, \ldots, x_m) \subset \mathbb{C}$, where $m\in \mathbb{N}$ and $\mathbb{C}$ denotes the field of complex numbers. The linearity of the functional equation means that the solutions form a vector space over $\mathbb{C}$. The idea of using spectral synthesis to find additive solutions of linear functional equations 
is natural due to the fundamental work \cite{KisLac25}. 
The key result says that spectral synthesis holds in any translation invariant closed linear subspace 
formed by additive mappings on a finitely generated subfield $K\subset \mathbb{C}$. 
Therefore such a subspace is spanned by so-called exponential monomials which can be given in terms of automorphisms of 
$\mathbb{C}$ and differential operators (higher order derivations), see also \cite{KisVin17} and \cite{KisVin17a}.

\subsection{Basic theoretical facts} Let $(G, *)$ be an Abelian group. By a \emph{variety} $V$ on $G$ we mean a translation invariant closed linear
subspace of $\mathbb{C}^G$, where $\mathbb{C}^G$ denotes the space of complex valued functions defined on $G$. 
The space of functions is equipped with the product topology. The translation invariance of the linear subspace provides that $f_g\colon  G \to \mathbb{C}$, $f_g(x)=f(g*x)$ is an element of $V$ for any $f\in V$ and $g\in G$. An additive mapping is a homomorphism of $G$ into the additive group of $\mathbb{C}$. The so-called {\it polynomials} are the elements of the algebra generated by the additive and constant functions. An {\it exponential mapping} is a nonzero (and, consequently, injective) homomorphism of $G$ into the multiplicative group of $\mathbb{C}$, i .e. $m\in \mathbb{C}^{G}$ such that $m(x*y)=m(x)\cdot m(y)$. An {\it exponential monomial} is the product of an exponential and a polynomial function. The finite sums of exponential monomials are called {\it polynomial-exponentials}. If a variety $V$ is spanned by exponential monomials then we say that {\it spectral synthesis} holds in $V$. If 
spectral synthesis holds in every variety on $G$ then spectral synthesis holds on $G$. 
Especially, \emph{spectral analysis} holds on $G$, i.e. every nontrivial variety contains an exponential; 
see Lemma 2.1 in \cite{KisLac25}.

To formulate the key result of \cite{KisLac25} let $G:=K^*$ (the multiplicative group of $K$) and consider the variety $V_a$ on $K^*$ consisting of the restriction of additive functions on $K$ (as an additive group) to $K^*$, i.e. 
\begin{equation}
\label{variety}
V_a=\{A |_{K^*} \ | \ A(x+y)=A(x)+A(y), \ \text{where}\ x\ \textrm{and}\ y\in K\}.
\end{equation}
It can be easily seen that if $A$ is an additive function then its translate $A_c(x)=A(cx)$ with respect to the multiplication by $c\in K^*$ is also additive. Theorem 4.3 in 
\cite{KisLac25} states that if the transcendence degree of $K$ over $\mathbb{Q}$ is 
finite\footnote{If $K$ is finitely generated over $\mathbb{Q}$ then it is automatically satisfied.} 
then spectral synthesis holds in every variety $V$ contained in $V_a$. 
By Theorem 3.4 in \cite{KisLac25}, the polynomials in $V$  correspond to mappings of the form $D(x)/x$ ($x\in K^*$), 
where $D$ is a differential operator on $K$. A \emph{differential operator} means the (complex) linear combination of 
finitely many mappings of the form $d_1 \circ \ldots \circ d_k$, where $d_1$, $\ldots$, $d_k$ are derivations on $K$. 
If $k=0$, then this expression is by convention the identity function. 
The exponentials in $V$ satisfy both
$m(x+y)=m(x)+m(y)$ and $m(x\cdot y)=m(x)\cdot m(y)$, where $x, y\in K$ and $m(0)=0$. 
Extending $m$ to an automorphism of $\mathbb{C}$ (see Lemma 4.1 in \cite{KisLac25}), 
a special expression for the elements (exponential monomials) spanning a variety $V$ contained in $V_a$ can be given: 
they are of the form $\varphi \circ D$, where $\varphi$ is the extension of an exponential 
$m\in V$ to an automorphism of $\mathbb{C}$ and $D$ is a differential operator on $K$; see Theorem 4.2 in \cite{KisLac25}. 

\subsection{Spectral synthesis in the variety generated by the solutions of equation \eqref{high_single}}  
Let $n\in \mathbb{N}$ and consider an additive mapping $A\colon K\to \mathbb{C}$ satisfying equation 
 \begin{equation}\label{high_singleagain}
  \sum_{i=0}^{n}(-1)^{i}\binom{n+1}{i}x^{i}A\left(x^{n+1-i}\right)=0 
  \qquad 
  \left(x\in K\right). 
 \end{equation}
First of all observe that the solutions of functional equation \eqref{high_singleagain} form a linear subspace. 
To get some information about the translates of the form $A_c(x)=A(cx)$ of the solutions we have to set the variable $x$ free by using a symmetrization process:
\begin{multline*}
  \Phi(x_{1}, \ldots, x_{n+1})
  =A(x_{1}\cdots x_{n+1})-\sum_{1\leq i\leq n+1}x_{i}A(x_{1}\cdots \hat{x}_{i}\cdots x_{n+1})
  \\+\sum_{i\leq i< j\leq n+1}x_{i}x_{j}A\left(x_{1}\cdots \hat{x}_{i}\cdots \hat{x}_{j}\cdots x_{n+1}\right)-
  \\ \cdots +(-1)^{n}\sum_{1\leq i\leq n+1}x_{1}\cdots \hat{x}_{i}\cdots x_{n+1}A(x_{i})
  \qquad 
  \left(x_{1}, \ldots, x_{n+1}\in K\right). 
 \end{multline*}
Due to equation (\ref{high_singleagain}), the diagonal $\Phi^{\ast}(x)=\Phi(x,\ldots,x)$
is identically zero. In the sense of Lemma \ref{mainfact}, the mapping $\Phi$ is also identically zero. By some direct computations 
\begin{multline*}
\Phi(x, \ldots, x, cx)\\
=A(c x^{n+1})-\binom{n}{1}xA(c x^n)-c xA(x^n)+\binom{n}{2}x^2A(c x^{n-1})+\binom{n}{1}c x^2A(x^{n-1})-\\
\binom{n}{3}x^3A(c x^{n-2})-\binom{n}{2}c x^3A(x^{n-2})
\\
+\ldots+(-1)^n x^{n}A(c x)+(-1)^n\binom{n}{n-1}c x^{n}A(x),
 \end{multline*}
\begin{multline*}
\Phi(x, \ldots, x, c)\\
=A(c x^{n})-\binom{n}{1}x A(c x^{n-1})-c A(x^n)+\binom{n}{2}x^2A(c x^{n-2})+\binom{n}{1}c xA(x^{n-1})-\\
\binom{n}{3}x^3A(c x^{n-3})-\binom{n}{2}c x^2A(x^{n-2})
\\
+\ldots+(-1)^n x^{n}A(c)+(-1)^n\binom{n}{n-1}c x^{n-1}A(x)
\end{multline*}
and, consequently, for any $c\in K^*$
$$\Phi(x, \ldots, x, cx)-x\Phi(x, \ldots, x, c)= \sum_{i=0}^{n+1}(-1)^{i}\binom{n+1}{i}x^{i}A\left(c x^{n+1-i}\right)
  \qquad 
  \left(x\in K\right), $$
	i.e. the translation invariant linear subspace generated by the solutions of equation (\ref{high_singleagain}) can be described by the family of equations
	 \begin{equation}\label{high_singleagaintranslated}
  \sum_{i=0}^{n+1}(-1)^{i}\binom{n+1}{i}x^{i}A\left(c x^{n+1-i}\right)=0 
  \qquad 
  \left(x\in K, c\in K^* \right).  
 \end{equation}

Let $c\in K^*$ and $x\in K$ be given and suppose that $f\colon K\to \mathbb{C}$ is the limit function\footnote{Since $K$ is countable, the limit can be taken in a pointwise sense.}
of the sequence $A_l$ of solutions of equation \eqref{high_singleagaintranslated}, that is for any $\varepsilon>0$ we have that
\[
\left |A_l(c x^{n+1-i})-f(c x^{n+1-i})\right | < \varepsilon \qquad (i=0, \ldots, n+1)
\]
provided that $l$ is large enough. Then 
\begin{multline*}
 \left| \sum_{i=0}^{n+1}(-1)^{i}\binom{n+1}{i}x^{i}f\left(c x^{n+1-i}\right)\right|
\\
=
\left|\sum_{i=0}^{n+1}(-1)^{i}\binom{n+1}{i}x^{i}f\left(c x^{n+1-i}\right)-\sum_{i=0}^{n+1}(-1)^{i}\binom{n+1}{i}x^{i}A_l\left(c x^{n+1-i}\right)\right| 
\\
\leq 
\sum_{i=0}^{n+1}\binom{n+1}{i}|x|^{i} \left |f\left(c x^{n+1-i}\right)-A_l\left(c x^{n+1-i}\right)\right |=\varepsilon \left(1+|x|\right)^{n+1}.
\end{multline*}

Therefore the space of solutions is a translation invariant closed linear subspace, i.e. it is a variety in $V_a$. 
Using the exponential element $m$ we have that
 \begin{multline*}
0=\sum_{i=0}^{n+1}(-1)^{i}\binom{n+1}{i}x^{i}m\left(c x^{n+1-i}\right)\\
=m(c)\sum_{i=0}^{n+1}(-1)^{i}\binom{n+1}{i}x^{i}m^{{n+1-i}}(x)=
m(c)(m(x)-x)^{n+1}
\end{multline*}
and, consequently, the exponential element must be the identity on the finitely ge\-ne\-ra\-ted field over $\mathbb{Q}$. 
This means that the space of the solutions is spanned by differential operators. 

\subsection{Spectral synthesis in the variety generated by the solutions of equation \eqref{gen_single}}  Let $n\in \mathbb{N}$ and consider an additive mapping $A\colon K\to \mathbb{C}$ satisfying equation 
 \begin{equation}\label{gen_singleagain}
  \sum_{i=0}^{n}a_{n+1-i}x^{i}A\left(x^{n+1-i}\right)=0 
  \qquad 
  \left(x\in K\right). 
 \end{equation}
First of all observe that the solutions of functional equation (\ref{gen_singleagain}) form a linear subspace. 
To get some information about the translates of the form $A_c(x)=A(c x)$ of the solutions we have to set the variable $x$ 
free by using a symmetrization process:
\begin{multline*}
  \Phi(x_{1}, \ldots, x_{n+1})=
 \sum_{i=0}^{n}a_{n+1-i}\dfrac{1}{\binom{n+1}{i}}\sum_{\mathrm{card}(I)=i}\left(\prod_{j\in I}x_{j}\right)\cdot 
 A\left(\prod_{k\in \left\{1, \ldots, n+1\right\}\setminus I}x_{k}\right) 
 \\
 \left(x_{1}, \ldots, x_{n+1}\in K\right). 
\end{multline*}
Then 
$$\Phi^{\ast}(x)=\Phi(x, \ldots, x)= \sum_{i=0}^{n}a_{n+1-i}x^{i}A\left(x^{n+1-i}\right)=0
 \qquad 
 \left(x\in K\right)
$$
because of equation \eqref{gen_singleagain}. In the sense of Lemma \ref{mainfact}, $\Phi$ is also identically zero. By some direct computations 
\[
\Phi(x, \ldots, x, cx)-x\Phi(x, \ldots, x, c)
=\dfrac{1}{n+1}\sum_{i=0}^{n} (n+1-i) a_{n+1-i}\left(x^iA(c x^{n+1-i})-x^{i+1}A(c x^{n-i})\right),
\]
that is the translation invariant linear subspace generated by the solutions of equation (\ref{gen_singleagain}) can be described by the family of equations
	 \begin{equation}\label{gen_singleagaintranslated}
  \sum_{i=0}^{n} (n+1-i) a_{n+1-i}\left(x^i A(c x^{n+1-i})-x^{i+1}A(c x^{n-i})\right)=0
	\qquad \left(x\in K, c\in K^* \right). 
 \end{equation}
By the same way as above we can prove that the space of the solutions of equation (\ref{gen_singleagaintranslated}) is a variety in $V_a$. Using the exponential element $m$ we have that
 $$0=m(c)(m(x)-x)\sum_{j=1}^{n+1} ja_j m^{j-1}(x)x^{n-(j-1)},$$
where $j=n+1-i$, $i=0, \ldots, n$. If $m(x)\neq x$ for some $x\in K^*$ then we can write that
$$0=\sum_{j=1}^{n+1} ja_j \left(\frac{m(x)}{x}\right)^{j-1}.$$
Therefore $m(x)/x$ is the root of the polynomial $\sum_{j=1}^{n+1} ja_j t^{j-1}$
and it has only finitely many different values. This is obviously a contradiction because for any $x\in K^*$, the function
\[
r\in \mathbb{Q}\longmapsto \frac{m(x+r)}{x+r}=\frac{m(x)+r}{x+r}
\]
provides infinitely many different values unless $m(x)=x$; note that $m(x)\neq x$ is equivalent to $m(x+r)\neq x+r$ for any rational number $r\in \mathbb{Q}$. Therefore we can conclude that the exponential element must be the identity on any finitely generated field over $\mathbb{Q}$ and the space of the solutions is spanned by differential operators.

\subsection{The solutions of equation \eqref{Eq7} in case of $n=3$} 
Let $n\in \mathbb{N}$, $n\geq 1$ be arbitrary and assume that the additive functions 
$f_{1}, \ldots, f_{n+1}\colon\allowbreak K\to \mathbb{C}$ satisfy equation
 \begin{equation}\label{Eq7again}
  \sum_{i=0}^{n}x^{i}f_{n+1-i}\left(x^{n+1-i}\right)=0 
  \qquad 
  \left(x\in K\right)
 \end{equation}
under the initial conditions $f_{i}(1)=0$ for all $i=1, \ldots, n+1$. In the first step we set the variable $x$ free by using a symmetrization process:
\[
  \Phi(x_{1}, \ldots, x_{n+1})
 =\sum_{i=0}^{n}\dfrac{1}{\binom{n+1}{i}}\sum_{\mathrm{card}(I)=i}\left(\prod_{j\in I}x_{j}\right) \cdot f_{n+1-i}\left(\prod_{k\in \left\{1, \ldots, n+1\right\}\setminus I}x_{k}\right) 
 \qquad 
 \left(x_{1}, \ldots, x_{n+1}\in K\right). 
\]
Then
$$\Phi^{\ast}(x)=\Phi(x, \ldots, x)= \sum_{i=0}^{n}x^{i}f_{n+1-i}\left(x^{n+1-i}\right)=0 
 \qquad 
 \left(x\in K\right)
$$
because of equation \eqref{Eq7again}. In the sense of Lemma \ref{mainfact}, $\Phi$ is also identically zero. 
We are going to investigate the explicit case of $n=3$, that is, equation 
 \begin{equation}\label{Eq7againspec}
  \sum_{i=0}^{3}x^{i}f_{4-i}\left(x^{4-i}\right)=0 
  \qquad 
  \left(x\in K\right). 
 \end{equation}
By some direct computations 
\begin{multline*}
\Phi(x,x,x,cx)
=f_4(cx^4)+\dfrac{3}{4}xf_3(cx^3)+\dfrac{1}{4}cxf_3(x^3)+\dfrac{1}{2}x^2f_2(cx^2)
\\
+\dfrac{1}{2}cx^2f_2(x^2)+\dfrac{1}{4} x^3f_1(cx)+\dfrac{3}{4}cx^3f_1(x),
 \end{multline*}
\begin{multline*}
\Phi(x,x,x,c)
=f_4(cx^3)+\dfrac{3}{4}xf_3(cx^2)+\dfrac{1}{4}cf_3(x^3)+\dfrac{1}{2}x^2f_2(cx)
\\
+\dfrac{1}{2}cxf_2(x^2)+\dfrac{1}{4} x^3f_1(c)+\dfrac{3}{4}cx^2f_1(x)
\end{multline*}
and, consequently, for any $c\in K^*$
\begin{multline*}
\Phi(x, x,x, cx)-x\Phi(x, x,x, c)
=f_4(cx^4)+x\left(\dfrac{3}{4}f_3-f_4 \right)(cx^3)
\\
+x^2\left(\dfrac{1}{2}f_2-\dfrac{3}{4}f_3 \right)(cx^2)
+x^3\left(\dfrac{1}{4}f_1-\dfrac{1}{2}f_2 \right)(cx)-\dfrac{1}{4}x^4f_1(c),
	\end{multline*}
 where $x\in K$,	i.e. we can formulate the family of equations
	 \begin{equation}\label{Eq7againspectranslated}
g_4(cx^4)+xg_3(cx^3)+x^2g_2(cx^2)+x^3g_1(cx)=x^4(g_1+g_2+g_3+g_4)(c) \qquad \left(x\in K\right),
 \end{equation}
where $c$ runs through the elements of $K^*$ and
$$\left[ {\begin{array}{c}
   g_1 \\
   g_2 \\
	g_3 \\
	g_4 \\
  \end{array} } \right]=\left[ {\begin{array}{cccc}
   1/4 & -1/2 & 0 & 0\\
   0 & \ \ 1/2 & -3/4 & 0\\
	0&0&\ \ 3/4 &-1 \\
	0& 0 &0 &\ 1\\
  \end{array} } \right]\left[ {\begin{array}{c}
   f_1 \\
   f_2 \\
	f_3 \\
	f_4
  \end{array} } \right].$$
The inverse formulas are 
$$f_4=g_4,\ f_3=\dfrac{4}{3}\left(g_3+g_4\right), \ f_2=2\left(g_2+g_3+g_4\right),\ f_1=4\left( g_1+g_2+g_3+g_4\right).$$
Taking $c=1$ we have that
	 \begin{equation}\label{Eq7againspeciterative}
g_4(x^4)+xg_3(x^3)+x^2g_2(x^2)+x^3g_1(x)=0   \qquad 
  \left(x\in K\right)
 \end{equation}
because of the initial conditions $f_1(1)=f_2(1)=f_3(1)=f_4(1)=0$. Therefore the space of the solutions is invariant under the action of the linear transformation represented by the matrix
$$M:=\left[ {\begin{array}{cccc}
   1/4 & -1/2 & 0 & 0\\
   0 & \ \ \ 1/2 & -3/4 & 0\\
	0&0&\ \ \ 3/4 &-1 \\
	0& 0 &0 &\ \ \ 1\\
  \end{array} } \right]=\frac{1}{4}\left[ {\begin{array}{cccc}
   1 & -2 & 0 & 0\\
   0 & \ \ \ 2 & -3 & 0\\
	0&0&\ \ \ 3 &-4 \\
	0& 0 &0 &\ \ \ 4\\
  \end{array} } \right].$$
As a MAPLE computation shows
\begin{verbatim}
with(LinearAlgebra);
M:=(1/4)*Matrix([[1,-2,0,0],[0,2,-3,0],[0, 0,3,-4], [0,0,0,4]]);
MatrixPower(M,n);
\end{verbatim}
	
$$ \left[ \begin {array}{cccc} {4}^{-n}&-{2}^{1-n}+2\,{4}^{-n}&{3}^{1+n}
{4}^{-n}-6\,{2}^{-n}+3\,{4}^{-n}&-4+12\,{3}^{n}{4}^{-n}-12\,{2}^{-n}+{
2}^{2-2\,n}\\ \noalign{\medskip}0&{2}^{-n}&-{3}^{1+n}{4}^{-n}+3\,{2}^{
-n}&6-12\,{3}^{n}{4}^{-n}+6\,{2}^{-n}\\ \noalign{\medskip}0&0&{3}^{n}{
4}^{-n}&-4+{2}^{2-2\,n}{3}^{n}\\ \noalign{\medskip}0&0&0&1\end {array}
 \right] 
 $$
and, consequently, 
$$\lim_{n\to \infty} M^n=\left[ {\begin{array}{cccc}
   0 & 0 & 0 & -4\\
   0 & 0 & 0 & \ \ \ 6\\
	0&0&0 &\ -4 \\
	0& 0 &0 &\ \ \ 1\\
	 \end{array} } \right].$$
	Therefore
	\begin{equation}\label{Eq7againspeciterativelimit}f_4(x^4)-4xf_4(x^3)+6x^2f_4(x^2)-4x^3f_4(x)=0   \qquad 
  \left(x\in K\right)
	\end{equation}
	and, by Theorem \ref{Phighmulti} (Corollary \ref{Prop3}), we can conclude that $f_4$ is a differential operator on any finitely generated field $K$. By taking the difference of equations (\ref{Eq7againspec}) and (\ref{Eq7againspeciterativelimit}) the number of the unknown functions can be reduced:
		\begin{equation}\label{Eq7againspecreduced}
		\tilde{f}_3(x^3)+x\tilde{f}_2(x^2)+x^2\tilde{f}_1(x)=0   \qquad 
  \left(x\in K\right),
	\end{equation}
	where
	$$\tilde{f}_3=f_3+4f_4,\ \tilde{f}_2=f_2-6f_4,\ \tilde{f}_1=f_1+4f_4.$$
	Repeating the process above we can conclude that $\tilde{f}_3$ is a differential operator on any finitely generated field $K$ and so on. 
	Note that it is an alternative way to prove Theorem \ref{Thm_mainmultivariate}/Corollary \ref{Thm_main} by using a descending process instead of the inductive argument. 
	
\vspace{1cm}

\noindent
\textbf{Acknowledgement.}
This paper is dedicated to the $65$\textsuperscript{th} birthday of Professor László Székelyhidi.

The research of the first author has been supported by the Hungarian Scientific Research Fund
(OTKA) Grant K 111651 and by the ÚNKP-4 New National Excellence Program of the Ministry of Human Capacities.
The work of the first and the third author is also supported by the
EFOP-3.6.1-16-2016-00022 project. The project is
co-financed by the European Union and the European
Social Fund. 
The second author was supported by the internal research project R-AGR-0500 of 
the University of Luxembourg and by the Hungarian Scientific Research Fund (OTKA) K 104178.

\bibliographystyle{plain}
\bibliography{ebanks}

\end{document}